\documentclass[a4paper,leqno,12pt]{amsart}

\usepackage{ulem}
\usepackage{float}
\usepackage{enumerate}
\usepackage{amssymb,amsthm,amsmath}
\usepackage{mathrsfs}
\usepackage[colorlinks,breaklinks]{hyperref}
\usepackage[margin=1in]{geometry}
\usepackage{braket}
\usepackage{comment}
\usepackage{tikz}
\usepackage{pgfplots}
\usepackage{graphicx}
\usepackage{comment}
\usepackage{caption}
\pgfplotsset{compat=1.11}
\usepgfplotslibrary{fillbetween}
\usetikzlibrary{intersections}

\usetikzlibrary{positioning,arrows}
\usepackage{capt-of}
\usepackage{caption}
\usepackage{float}
\usetikzlibrary{shapes}
\usetikzlibrary{plotmarks}
\tikzset{
  state/.style={circle,draw,minimum size=6ex},
  arrow/.style={-latex, shorten >=1ex, shorten <=1ex}}

\theoremstyle{plain}
\newtheorem{prop}{Proposition}[section]
\newtheorem{lemma}[prop]{Lemma}
\newtheorem{theorem}[prop]{Theorem}

\theoremstyle{definition}

\theoremstyle{remark}

\newcommand{\C}{\mathbb{C}}
\newcommand{\R}{\mathbb{R}}

\newcommand{\ZZ}{\mathbb{Z}}
\newcommand{\CC}{\mathbb{C}}

\newcommand{\wt}{\widetilde}
\newcommand{\wh}{\widehat}

\newcommand{\pmat}[4]{\begin{pmatrix} #1&#2\\#3&#4\end{pmatrix}}

\DeclareMathOperator{\tr}{tr}
\DeclareMathOperator{\rank}{rank}
\DeclareMathOperator{\supp}{supp}

\DeclareMathOperator{\sign}{sign}
\DeclareMathOperator{\vol}{vol}

\DeclareMathOperator{\Res}{Res}

\newcommand*{\defeq}{\stackrel{\text{def}}{=}}

\allowdisplaybreaks

\makeatletter
\@namedef{subjclassname@2020}{%
  \textup{2020} Mathematics Subject Classification}
\makeatother

\usetikzlibrary{patterns}

\begin{document}

\title[]{Low-lying resonances for infinite-area hyperbolic surfaces with long closed geodesics}

\author[L.\@ Soares]{Louis Soares}
\email{louis.soares@gmx.ch}

\begin{abstract}
We consider sequences $(X_n)_{n\in \mathbb{N}}$ of coverings of convex cocompact hyperbolic surfaces $X$ with Euler characterictic $\chi(X_n)$ tending to $-\infty$ as $n\to \infty.$ We prove that for $n$ large enough, each $X_n$ has an abundance of ``low-lying'' resonances, provided the length of the shortest closed geodesic on $X_n$ grows sufficiently fast. When applied to congruence covers we obtain a bound that improves upon a result of Jakobson, Naud, and the author in \cite{JNS}. Our proof uses the wave 0-trace formula of Guillop\'{e}--Zworski \cite{GZ99} together with specifically tailored test-functions with rapidly decaying Fourier transform.
\end{abstract}
\maketitle

\section{Introduction and Statement of Results}
In the last few decades, much effort has been dedicated to the study of quantum resonances, see the expository article of Zworski \cite{Zworski99} for a general introduction. In this paper, we are interested in infinite-area, ``convex cocompact'' hyperbolic surfaces, that is, quotients $X = \Gamma\backslash \mathbb{H}^2$ that can be decomposed as
\begin{equation}\label{decompX}
X = K \cup F_1 \cup \cdots \cup F_m,
\end{equation}
where $K$ is a compact surface with geodesic boundary and $F_1, \dots, F_m$ are finitely many funnel ends glued to $K$. We refer to Borthwick's book \cite{Borthwick_book} for an introduction to the spectral theory of infinite-area hyperbolic surfaces. The classical $L^2$-spectrum for the Laplace--Beltrami operator $\Delta_X$ on $X$ is rather sparse. In particular, $\Delta_X$ has only finitely many $L^2$-eigenvalues. The main spectral data of interest are the \textit{resonances} of $X$, which are defined as the poles of the meromorphic continuation of the resolvent operator 
\begin{equation}\label{res_continued}
R_X(s)\defeq (\Delta_X - s(1-s))^{-1} \colon C_c^\infty(X)\to C^\infty(X).
\end{equation}
In the sequel, we denote by $\mathcal{R}(X)$ the set of resonances for $X$, repeated according to multiplicities, see §\ref{sec:prelims} for details. We denote by $\delta\in (0,1)$ the Hausdorff dimension of the limit set of $X$, which is equal to the \textit{critical exponent} of the Poincar\'e series of $\Gamma$. There are infinitely many resonances, all of which are contained in the half-plane $\mathrm{Re}(s)\leqslant \delta$. The vertical line $\mathrm{Re}(s)=\delta$ has exactly one simple resonance at $s=\delta.$ In general, it is a notoriously difficult problem to precisely locate resonances.

From the point of view of physics, resonances determine the ways in which $X$ can ``vibrate''. Indeed, each $s\in \mathcal{R}(X)$ corresponds to a stationary solution
\begin{equation}
\psi_s(x,t) = e^{(s-\frac{1}{2})t} \phi_s(x), \quad x\in X,\, t\in (0,\infty)
\end{equation}
of the automorphic wave equation
\begin{equation}
\left( \Delta_X + \partial_t - \frac{1}{4} \right)\psi_s = 0.
\end{equation}
In this view, the real part of $s$ determines the decay amplitude and the imaginary part determines the oscillation frequency of the solution $\psi_s.$ 

Resonances not corresponding to $L^2$-eigenvalues are precisely those that lie in the half-plane $\mathrm{Re}(s)\leqslant \frac{1}{2}.$ We study the counting function
$$
\mathsf{N}(\sigma, T; X) \defeq \#\left\{ s\in \mathcal{R}(X) : \text{$ \sigma \leqslant \mathrm{Re}(s)\leqslant \frac{1}{2} $ and $ \vert \mathrm{Im}(s)\vert < T  $} \right\},
$$
where resonances are counted with multiplicities. Guillop\'{e}--Lin--Zworski~\cite{Guillope_Lin_Zworski} proved the upper fractal Weyl bound
$$
\mathsf{N}(\sigma, T; X) \ll T^{1+\delta}, \quad T\to \infty.
$$
The first estimate of this form was established by Sj\"ostrand in his pioneering work on semi-classical Schr\"odinger operators \cite{Sjoestrand}. The following non-exhaustive list of papers \cite{Lu_Sridhar_Zworski,Zworski99,Guillope_Lin_Zworski, Datchev_Dyatlov, NSZ_fractal, faure2022fractal, SjoZwo, PWBKSZ, NZ_fractal, NoRu07, Naud_Pohl_Soares} provides similar fractal Weyl bounds. All of these results support a well-known conjecture in open chaotic systems and quantum chaos, known as the \textit{fractal Weyl law}. In our setting, this conjecture says that for all $\sigma\in \mathbb{R}$ negative enough we should have
$$
\mathsf{N}(\sigma, T; X) \asymp T^{1+\delta}, \quad T\to \infty.
$$
For numerical evidence for this law we refer to \cite{Borthwick_numerical} and the references therein. While upper bounds for counting functions of resonances have become well understood, lower bounds of the same strength have remained elusive. Guillop\'{e}--Zworski \cite{GZ99} proved that for any $0<\epsilon < \frac{1}{2}$ and $\sigma < -\frac{1}{2 \epsilon}$ we have
$$
\mathsf{N}(\sigma, T; X) \neq O( T^{1-\epsilon} ).
$$
Jakobson--Naud \cite{JN2010} proved a lower bound with an explicit dependence on $\delta$, but the growth rate is only logarithmic unless $\Gamma$ is an arithmetic group. 

The aim of this paper is to give lower bounds for $\mathsf{N}(\sigma, T; X)$ for small values of $T$ and $\sigma$ in terms of geometric quantities of $X$. We denote by $\ell_0(X)$ the length of the shortest closed geodesic on $X$ and by $\vol_0(X)$ the volume of the compact core of $X$. Using the Gauss-Bonet formula for the area of hyperbolic polygons we deduce that if $X$ has a decomposition as in \eqref{decompX}, then
$$
\vol_0(X) = \vol(K) = -2\pi (m-1) = -2\pi \chi(X),
$$ 
where $\chi(X)$ is the Euler characteristic of $X$, see also \cite[Lemma 10.3]{Borthwick_book}. Our main result is the following:

\begin{theorem}[Main theorem]\label{thm:main_theorem} 
Let $X$ be a non-elementary, infinite-area, convex cocompact hyperbolic surface with $\delta > \frac{1}{2}$. Let $(X_n)_{n\in \mathbb{N}}$ be a sequence of finite-degree coverings of $X$ such that the following hold true as $n\to \infty$:
\begin{itemize}
\item $\vol_0(X_n) \to \infty$ and
\item there exists some $A>0$ such that $\ell_0(X_n)\geqslant (A-o(1)) \log( \vol_0(X_n) ).$
\end{itemize}
Then for all $\alpha > 1$ and $ \epsilon, \epsilon'>0$ there are constants $c = c(\alpha,\epsilon,\epsilon')>0$ and $n_0 = n_0(\alpha,\epsilon,\epsilon')$ such that for all $n\geqslant n_0$ we have
$$
\mathsf{N}\left( \delta - \frac{1}{A} - \epsilon', (\log \log ( \vol_0(X_n)))^\alpha; X_n \right) \geqslant  c  \vol_0(X_n)^{A( \delta - \frac{1}{2}-\epsilon)} .
$$
\end{theorem}

We remark that any finite-degree cover of $X$ has the same critical exponent $\delta,$ that is, the point $s=\delta$ is a common leading  resonance for all finite-degree covers of $X$.

Our next goal is to specialize the main result to congruence covers of arithmetic surfaces. When $\Gamma$ a subgroup of $\mathrm{SL}_2(\mathbb{Z})$ and $n\in \mathbb{N}$, we define the principal congruence subgroup of $\Gamma$ of level $n$ by
\begin{align*}
\Gamma(n) &\defeq \left\{ \text{$\gamma\in \Gamma$ : $\gamma\equiv I$ mod $n$} \right\},
\end{align*}
and we write $X(n) = \Gamma(n)\backslash\mathbb{H}^2$ for the associated covering. Infinite-area congruence covers have attracted a lot of attention due to their applications in number theory and graph theory, see \cite{Gamburd1,BGS,OhWinter}. In these papers, uniform spectral gaps (that is, uniform resonance-free strips) are established. Our next result is a somewhat complementary statement, stating that congruence covers have an abundance of low-lying resonances. More concretely, we have the following:

\begin{theorem}\label{thm:congr}
Assume that $X=\Gamma\backslash \mathbb{H}^2$ is as in Theorem \ref{thm:main_theorem} and that $\Gamma \subset \mathrm{SL}_2(\mathbb{Z})$. Then for all $\alpha > 1$ and $ \epsilon,\epsilon' >0$ there are $c = c(\alpha,\epsilon,\epsilon')>0$ and $n_0 = n_0(\alpha,\epsilon,\epsilon') \in \mathbb{N}$ such that for all positive integers $n$ with $\gcd(n,n_0)=1$ we have
\begin{equation}\label{estimatecongr}
\mathsf{N}\left( \delta - \frac{3}{4} - \epsilon', (\log \log n)^\alpha; X(n) \right) \geqslant c n^{4 ( \delta-\frac{1}{2} )-\epsilon}.
\end{equation}
\end{theorem}

This result strengthens Theorem 1.4 in the work of Jakobson, Naud, and the author \cite{JNS} in several ways. 

The proof of our main Theorem \ref{thm:main_theorem} uses the wave 0-trace formula of Guillop\'{e}--Zworski \cite{GZ99} together with specifically tailored test-functions $\varphi$ with rapidly decaying Fourier transform $ \widehat{\varphi}(z) $ as $\vert \mathrm{Re}(z)\vert\to \infty$. The construction of these test-functions uses the same harmonic analysis ingredient as in \cite{JNS}, which in the present paper corresponds to Lemma \ref{lem:test_function} below. Theorem \ref{thm:congr} can be deduced from the main theorem in conjunction with lower bounds for the minimal length of closed geodesics on congruence covers. We conjecture that Theorems \ref{thm:main_theorem} and \ref{thm:congr} hold true for \textit{all} geometrically finite hyperbolic surfaces. The restriction to the convex cocompact case is needed because we also require the following bound for the global counting function in \cite{Pohl_Soares}:
$$
\mathsf{N}(r;X_n) \defeq \#\{ s\in \mathcal{R}(X) : \vert s\vert \leqslant r \} \ll \vol_0(X_n) r^2,
$$
where the implied constant depends only on the base surface $X$. This is currently only known for covers of convex cocompact surfaces.

The paper is organized as follows: in Section \ref{sec:prelims} we gather the necessary preliminaries needed for our main Theorem \ref{thm:main_theorem}, which is then proven in Section \ref{sec:proof1}. In Section \ref{sec:proof2} we prove Theorem \ref{thm:congr}.\\

\paragraph{\textbf{Notation}} 
We write $f (x)\ll g(x)$ or $ f (x) = O(g(x)) $ interchangeably to mean that there exists an implied constant $C>0$ such that $\vert f (x)\vert \leqslant  C \vert g(x)\vert$ for all $x \geqslant C$. We write $f(x)\asymp g(x)$ if $g(x)\ll f (x)\ll g(x)$. We write $f (x)\ll_y g(x)$, $ f (x) = O_y(g(x)) $, or $ f (x) \asymp_y g(x)$ to emphasize that the implied constant depends on $y$. We write $f(x)=o(1)$ to mean $f(x)\to 0$ as $x\to \infty$. All the implied constants are allowed to depend on base surface $X$, which we assume to be fixed throughout. We define the Fourier transform of a function $\psi\in C_c^\infty(\mathbb{R})$ as usual by
\begin{equation}\label{defi:fourier}
\widehat{\psi}(z) \defeq \int_{-\infty}^{\infty} \psi(x)e^{-izx} dx, \quad z\in \mathbb{C}.
\end{equation}

\section{Preliminaries}\label{sec:prelims}

\subsection{Hyperbolic geometry} Let us recall some basic facts about hyperbolic surfaces, referring the reader to \cite{Borthwick_book} for a comprehensive discussion. One of the standard models for the hyperbolic plane is the Poincar\'{e} half-plane
$$
\mathbb{H}^2=\{ x+iy\in \C\ :\ y>0\} 
$$
endowed with its standard metric of constant curvature $-1$,
$$
ds^2=\frac{dx^2+dy^2}{y^2}.
$$ 
The group of orientation-preserving isometries of $(\mathbb{H}^2, ds)$ is isomorphic to $\mathrm{PSL}_2(\R)$. It acts on the extended complex plane $\overline{\C} = \C \cup \{ \infty\}$ (and hence also on $\mathbb{H}^2$) by M\"{o}bius transformations
$$
\gamma=\pmat{a}{b}{c}{d}\in \mathrm{PSL}_2(\R),\quad z\in\overline{\C} \Longrightarrow  \gamma(z) = \frac{az+b}{cz+d}.
$$
An element $\gamma\in \mathrm{PSL}_2(\R)$ is either
\begin{itemize}
\item \textit{hyperbolic} if $\vert \mathrm{tr}(\gamma)\vert >2$, which implies that $\gamma$ has two distinct fixed points on the boundary $\partial \mathbb{H}^2$,
\item \textit{parabolic} if $\vert \mathrm{tr}(\gamma)\vert <2$, which implies that $\gamma$ has precisely one fixed point on $\partial \mathbb{H}^2$, or
\item \textit{elliptic} if $\vert \mathrm{tr}(\gamma)\vert = 2$, which implies that $\gamma$ has precisely one fixed point in the hyperbolic plane $\mathbb{H}^2$.
\end{itemize}

\subsection{Hyperbolic surfaces and Fuchsian groups}
Every hyperbolic surface $X$ is isometric to a quotient $\Gamma\backslash \mathbb{H}^2$, where $\Gamma$ is a \textit{Fuchsian} group, that is, a discrete subgroup $\Gamma\subset\mathrm{PSL}_2(\R)$. A Fuchsian group $\Gamma$ is called
\begin{itemize}
\item \textit{torsion-free} if it contains no elliptic elements,
\item \textit{non-cofinite} if the quotient $\Gamma\backslash \mathbb{H}^2$ has infinite-area,
\item \textit{non-elementary} if it is generated by more than one element, and
\item \textit{geometrically finite} if it is finitely generated, which is equivalent with $\Gamma\backslash \mathbb{H}^2$ being geometrically and topologically finite.
\end{itemize}

All the Fuchsian groups $\Gamma$ considered in this paper satisfy all the above conditions. The \textit{limit set} $\Lambda$ of $X$ is defined as the set of accumulation points of all orbits of the action of $\Gamma$ on $\mathbb{H}^2$. Under the above conditions, the limit set is a Cantor-like fractal subset of the boundary $\partial \mathbb{H}^2 \cong \mathbb{R}\cup \{ \infty \}$ with Hausdorff dimension $\delta$ lying strictly between $0$ and $1$.

Furthermore, $\Gamma$ is called \textit{convex cocompact} if it is finitely generated and if it contains neither parabolic nor elliptic elements. This is equivalent with the \textit{convex core} of $X = \Gamma\backslash \mathbb{H}^2$ being compact, or with $X$ being geometrically finite and having no cusps.

\subsection{Spectral theory of infinite-area hyperbolic surfaces}\label{sec:specth}
Let us review some aspects of the spectral theory of infinite-area hyperbolic surfaces. We refer the reader to \cite{Borthwick_book} for an in-depth account of the material given here. The $ L^{2} $-spectrum of the Laplace--Beltrami operator $\Delta_X $ on an infinite-area hyperbolic surface $X$ is rather sparse and was described by Lax--Phillips \cite{Lax_Phillips_I} and Patterson \cite{Patterson} as follows:
\begin{itemize}
\item The absolutely continuous spectrum is equal to $ [1/4, \infty) $.

\item The pure point spectrum is finite and contained in the interval $(0,1/4)$. In particular, there are no eigenvalues embedded in the continuous spectrum.

\item If $\delta \leqslant 1/2$ then the pure point spectrum is empty. If $\delta>1/2$ then $\lambda_0(X) = \delta(1-\delta)$ is the smallest eigenvalue.
\end{itemize}
In view of these facts, the resolvent operator
$$
R_X(s) \defeq \big( \Delta_X - s(1-s)\big)^{-1}\colon L^2(X) \to L^2(X)
$$
is defined for all $s\in\CC$ with $ \mathrm{Re}(s)>1/2$ and $s(1-s)$ not being an $L^2$-eigenvalue of $\Delta_X$. From Guillop\'{e}--Zworski \cite{GuiZwor2} we know that the resolvent extends to a meromorphic family
\begin{equation}\label{resolvent_continued}
R_X(s) \colon C_{c}^{\infty}(X) \to C^{\infty}(X)
\end{equation}
on $\CC$ with poles of finite rank. The poles of $R_X(s)$ are called the \textit{resonances} of $ X $; the multiplicity $m_X(\zeta) $ of a resonance $\zeta$ is defined as the rank of the residue operator of $R_X(s)$ at $s=\zeta$:
$$
m_X(\zeta) \defeq \rank \Res_{s=\zeta} R_X(s).
$$
We denote by $\mathcal{R}(X)$ the multiset of resonances of $X$ repeated according to multiplicities. Resonances are contained in the half-plane $\mathrm{Re}(s)\leqslant  \delta$, with no resonances on the vertical line $\mathrm{Re}(s)=\delta$ other than a simple resonance at $s=\delta$.

Note that resonances $s$ on the half-plane $\mathrm{Re}(s)>\frac{1}{2}$ correspond one-to-one to eigenvalues of $\Delta_X$ via the equation $\lambda = s(1-s)$. In the present paper we focus on the case $\delta  > \frac{1}{2}$. It is convenient to define the multiset
$$
\Omega(X) \defeq \left\{ s\in \left(\frac{1}{2},\delta\right] : \text{$\lambda = s(1-s)$ is an $L^2$-eigenvalue for $X$} \right\},
$$
where each $s$ is repeated according to its multiplicity $m_X(s)$.

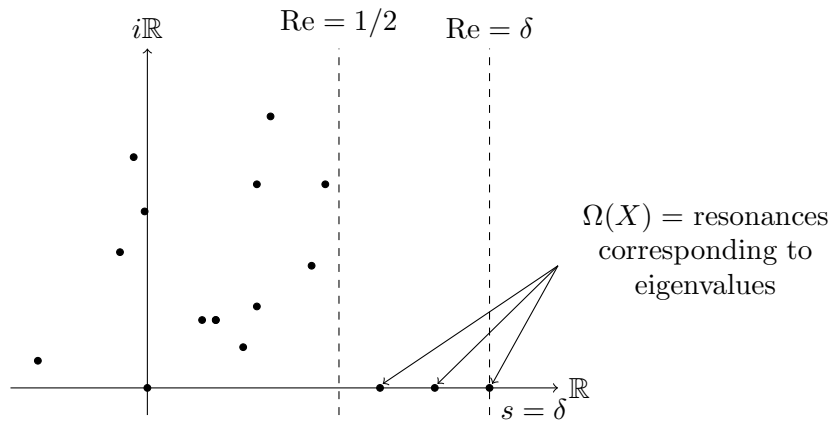
\begin{figure}[H]
\centering
\captionsetup{justification=centering}
\begin{tikzpicture}[xscale=1.8, yscale=1.8]
    \draw [->] (-1,0) -- (3,0) node [right, font=\small]  {$  \mathbb{R} $};
    \draw [->] (0,-0.2) -- (0,2.5) node [above, font=\small] {$ i \mathbb{R} $};
    \draw[dashed] (2.5,-0.2) -- (2.5,2.5) node [above, font=\small] {$ \mathrm{Re}= \delta $};
	\draw[dashed] (1.4,-0.2) -- (1.4,2.5) node [above, font=\small] {$ \mathrm{Re} = 1/2 $};
    \filldraw (2.5,0) circle (0.7pt) node[below right, font=\small] {$s=\delta$};
    \filldraw (1.7,0) circle (0.7pt);
    \filldraw (2.1,0) circle (0.7pt);
	\filldraw (0,0) circle (0.7pt);
	\filldraw (0.8,0.6) circle (0.7pt);
	\filldraw (0.4,0.5) circle (0.7pt);
	\filldraw (0.7,0.3) circle (0.7pt);
	\filldraw (-0.2,1) circle (0.7pt);
	\filldraw (-0.8,0.2) circle (0.7pt);
	\filldraw (0.9,2) circle (0.7pt);
	\filldraw (-0.1,1.7) circle (0.7pt);
	\filldraw (0.8,1.5) circle (0.7pt);
	\filldraw (-0.02,1.3) circle (0.7pt);
	\filldraw (0.5,0.5) circle (0.7pt);
	\filldraw (1.3,1.5) circle (0.7pt);
	\filldraw (1.2,0.9) circle (0.7pt);
	\filldraw (0.5,0.5) circle (0.7pt);
   	\draw (3,1) node[right, font=\small] {\begin{tabular}{c}
    $\Omega(X) =$ resonances \\
    corresponding to \\ 
    eigenvalues
\end{tabular}};
	\draw[->] (3,0.9) -- (2.52,0.03);
   	\draw[->] (3,0.9) -- (1.72,0.03);
   	\draw[->] (3,0.9) -- (2.12,0.03);
\end{tikzpicture}
\caption{Distribution of resonances for infinite-area $ \Gamma\backslash\mathbb{H}^2 $ in the case $ \delta > \frac{1}{2} $}
\label{fig:resonance_infinite_area}
\end{figure}

\subsection{The global resonance counting function} The distribution of resonances and eigenvalues of finite-area hyperbolic surfaces is fairly well-understood. In particular, the classical Selberg trace formula leads to a Weyl law for the global resonance counting function of the form
$$
\mathsf{N}(r;X) \defeq \#\{ s\in \mathcal{R}(X) : \vert s\vert \leqslant r \} \sim \frac{\vol(X)}{2\pi} r^2, \quad r\to \infty,
$$
proven in this case by M\"uller \cite{Mueller_scattering} and Parnovski \cite{Par95}. Here and henceforth, resonances are counted with multiplicities. In the infinite-area case, much less is known. Guillop\'e and Zworski \cite{GZ_upper_bounds,GZ_scattering_asympt} proved that for any geometrically finite hyperbolic surface $X$ there are constants $C_i =C_i(X)>0$, $i=1,2$, and $r_0=r_0(X)$ such that for all $r>r_0$ we have
\begin{equation}\label{ordergrowth}
C_1 r^2 \leqslant \mathsf{N}(r;X)\leqslant C_2 r^2. 
\end{equation}
In this paper we need to understand the behavior of $\mathsf{N}(r;X_n)$ for families of coverings $(X_n)_{n\in \mathbb{N}}$ of $X$ as $\vol_0(X_n) \to \infty.$ Unfortunately, the proof method in \cite{GZ_upper_bounds,GZ_scattering_asympt} yields only ineffective constant $C_1$ and $C_2$ in \eqref{ordergrowth} with no clear dependence on the geometry of $X$. However, when $X$ is a convex cocompact surface, it is known from Pohl--Soares \cite{Pohl_Soares} that there are constants $r_0 > 0$ and $C>0$, depending only on $X$, such that for any family $(X_n)_{n\in \mathbb{N}}$ of finite-degree coverings of $X$ and for all $r\geqslant r_0$ we have
\begin{equation}\label{pohl_soares}
C^{-1} \vol_0(X_n) r^2  \leqslant \mathsf{N}(r;X_n) \leqslant C \vol_0(X_n) r^2.
\end{equation}
This result follows essentially from an upper bound for \textit{twisted} Selberg zeta functions $Z_\Gamma(s,\rho)$ (sometimes also known as $L$-functions): there exists some $C>0$, depending only on $X = \Gamma\backslash\mathbb{H}^2$, such that for all finite-dimensional, unitary representations $\rho\colon \Gamma\to \mathrm{U}(V)$, we have
\begin{equation}\label{zetabound}
\log \vert Z_\Gamma(s,\rho)\vert \leq C \dim(\rho) \langle s \rangle^2, \quad  \langle s \rangle \defeq \sqrt{1+\vert s\vert^2}.
\end{equation}
We will actually only use the upper bound in \eqref{pohl_soares} in the present paper. This result generalizes and improves upon an earlier result of Jakobson--Naud \cite{JN} who have studied the specific case of congruence covers. Unfortunately, the estimate \eqref{zetabound}, which relies on the thermodynamic formalism for Selberg zeta functions, has so far only been established for convex cocompact groups $\Gamma.$ We should also mention the sharp upper bound for $\mathsf{N}(r;X)$ in terms of geometric quantities of $X$ due to Borthwick \cite{sharpbounds}.

\subsection{A harmonic analysis result}
We now state and re-prove a harmonic analysis result from the paper \cite{JNS}. Roughly speaking, it says that we can find a smooth and compactly supported function $\psi$ whose Fourier transform $\widehat{\psi}(z)$ has almost exponential decay as $\vert\mathrm{Re}(z)\vert\to \infty$. This can be seen as the complex analogue of the classical Beurling--Malliavin Multiplier theorem. Our proof is an almost verbatim repetition of the proof given in \cite{JNS} which in turn closely follows the construction given in \cite[Chapter~5, Lemma~2.7]{Katznelson}.

\begin{lemma}[Lemma 6.2 in \cite{JNS}]\label{lem:test_function} For all $ \beta > 1 $ there are constants $c_1 >0$ and $c_2  >0$ and a positive smooth function $ \psi $ supported in $[-1,1]$ such that
$$
\vert \widehat{\psi}(z) \vert \leqslant c_{1} e^{\vert \mathrm{Im}(z)\vert} \exp\left( -c_{2} \frac{\vert \mathrm{Re}(z)\vert }{(\log \vert \mathrm{Re}(z)\vert)^{\beta}} \right)
$$
for all $z\in \mathbb{C}$ with $\vert \mathrm{Re}(z)\vert\geqslant  2 .$
\end{lemma}

\begin{proof}
Let $ (\mu_j)_{j\in \mathbb{N}} $ be a sequence of positive real numbers summing up to $1$:
\begin{equation}\label{add_up_to_one}
\sum_{j=1}^{\infty} \mu_j = 1.
\end{equation}
For $ k\in \mathbb{Z} $ and $ n\in \mathbb{N} $ set
\begin{equation*}
\varphi_n(k) = \prod_{j=1}^{n} \frac{\sin(\mu_j k)}{\mu_j k}, \quad \varphi(k) = \prod_{j=1}^{\infty} \frac{\sin(\mu_j k)}{\mu_j k}
\end{equation*}
and consider the corresponding Fourier series
\begin{equation*}
\psi_{n}(x) = \sum_{k \in \ZZ} \varphi_n(k) e^{ikx}, \quad \psi(x) = \sum_{k \in \ZZ} \varphi(k) e^{ikx}.
\end{equation*}
The rapid decay of $ \varphi(k) $ implies that $ \psi $ is a $ C^{\infty} $-function on the interval $ [-2\pi, 2\pi]. $ On the other hand, $\psi_n$ converges uniformly to $\psi$ as $n\to \infty$. Moreover, we have
\begin{equation}\label{faltung}
\psi_n = g_1 \ast g_2 \ast \ldots \ast g_n,
\end{equation}
where $\ast$ denotes convolution operator and each $g_j$ is defined as
\begin{equation*}
g_j(x) = \begin{cases} \frac{2\pi}{\mu_j} &\text{ if } \vert x\vert \leqslant \mu_j, \\ 0 &\text{ else. }  \end{cases}
\end{equation*}
From \eqref{faltung} and \eqref{add_up_to_one} we deduce that $\psi$ is a positive function with support in $ [-1,1] $. We extend $ \psi $ outside $[-1,1]$ by zero. 

In what follows, we write $x=\mathrm{Re}(z).$ Integration by parts and Cauchy--Schwarz Inequality yield for all $ z\in \mathbb{C} $ with $ \vert x\vert \geqslant  2 $ and all $ m\geqslant 1 $ the bound
\begin{equation}\label{Cauchy_Schwarz_application}
\vert \widehat{\psi}(z)\vert \leqslant \frac{e^{\vert \mathrm{Im}(z)\vert}}{\vert x\vert^m}\Vert \psi^{(m)}\Vert_{L^2(-1,1)}.
\end{equation}
Using Plancherel's Theorem, we obtain
\begin{equation}\label{plancherel}
\Vert \psi^{(m)}\Vert_{L^2(-1,1)}^2 = \sum_{k \in \ZZ} k^{2m}\varphi(k)^2\leqslant C \prod_{j=1}^{m+1}\mu_j^{-2},
\end{equation}
where $C>0$ is some absolute constant. Combining \eqref{Cauchy_Schwarz_application} and \eqref{plancherel} gives 
\begin{equation}\label{combi_cauchy_plancherel}
\vert \widehat{\psi}(z)\vert \leqslant C \frac{e^{\vert \mathrm{Im}(z)\vert}}{\vert x\vert^m} \prod_{j=1}^{m+1}\mu_j^{-1}.
\end{equation}
Now fix $ \kappa > 0 $ and we let $ (\wt\mu_j)_{j\in \mathbb{N}} $ be the sequence defined by
\begin{equation}
\wt\mu_1 = 1, \quad \wt\mu_j = \frac{1}{j \log(j)^{1+\kappa}} \text{ for } j\geqslant   2.
\end{equation}
Note that the sequence $(\wt\mu_j)_{j\in \mathbb{N}}$ converges for all $\kappa>0$, so we may consider the normalized sequence $ \mu_j = c \wt \mu_j $, where $ c = c(\kappa) $ is chosen so that the $ \mu_j $'s add up to one. Using \eqref{combi_cauchy_plancherel} we get
\begin{align*}
\vert \widehat{\psi}(z)\vert &\leqslant C\cdot  c^m\frac{e^{\vert \mathrm{Im}(z)\vert}}{\vert x\vert^m} \prod_{j=2}^{m+1} j \log(j)^{1+\kappa}\\
&\leqslant \frac{e^{\vert \mathrm{Im}(z)\vert}}{\vert x\vert^m} m! e^{(1+\kappa) m \log\log (m) + O(m)}.
\end{align*}
Using Stirling's formula for the factorial $m! =(1+o(1))\sqrt{2\pi m}  (\frac{m}{e})^m$ and choosing 
\begin{equation*}
m = \left\lfloor \frac{\vert x\vert}{(\log \vert x\vert)^{1+2\kappa}} \right\rfloor 
\end{equation*}
yields (after some calculations)
\begin{equation}
\vert \widehat{\psi}(z)\vert \leqslant e^{\vert \mathrm{Im}(z)\vert} \exp\left( - \kappa \frac{(\log \log \vert x\vert) \vert x\vert}{(\log \vert x\vert)^{1+2\kappa}} ( 1+o(1) ) \right), \quad x\to \infty.
\end{equation}
This completes the proof of Lemma \ref{lem:test_function} with $\beta = 1+2\kappa$.
\end{proof}

\section{Proof of Theorem \ref{thm:main_theorem}}\label{sec:proof1}
We now give the proof of our main Theorem \ref{thm:main_theorem}. Let $X$ be an non-elementary, infinite-area, convex cocompact hyperbolic surface with critical exponent $\delta > \frac{1}{2}$, and let $(X_n)_{n\in \mathbb{N}}$ be a family of finite-degree coverings of $X$. Recall that we consider the resonance counting function
$$
\mathsf{N}(\sigma, t; X) \defeq \#\left\{ s\in \mathcal{R}(X) : \text{$ \sigma \leqslant \mathrm{Re}(s)\leqslant \frac{1}{2} $ and $ \vert \mathrm{Im}(s)\vert < t  $} \right\}.
$$
The starting point of our proof is the wave 0-trace formula of Guillop\'{e}--Zworski \cite{GZ99} (see also \cite[Corollary 11.5]{Borthwick_book}). When applied to non-elementary, convex cocompact, hyperbolic surfaces $Y$ this formula takes the following form: for all test-functions $\varphi\in C_c^\infty((0,\infty))$ we have
\begin{equation}\label{startingPoint}
 \sum_{s\in \mathcal{R}(Y)} \wh\varphi\left( i \left( s-\frac12\right) \right) = -\frac{\vol_0(Y)}{4\pi} \int_{-\infty}^{\infty} \frac{\cosh(\frac{t}{2})}{\sinh(\frac{t}{2})^2} \varphi(t)dt + \sum_{\ell\in  \mathcal{L}(Y)}\sum_{k=1}^\infty \frac{\ell \varphi(k\ell)}{2 \sinh(\frac{k \ell}{2})},
\end{equation}
where 
\begin{itemize}
\item $\mathcal{R}(Y)$ is the set of resonances of $Y$, repeated with multiplicities,
\item $\mathcal{L}(Y)$ is the length spectrum, that is, the set of lengths of the primitive periodic geodesics on $ Y $ repeated with multiplicities,
\item $\vol_0(Y)$ is the 0-volume of $Y$, which equals the volume of the compact core of $Y$, and
\item $\widehat{\varphi}$ is the Fourier transform of $\varphi$ defined by \eqref{defi:fourier}.
\end{itemize}

By substituting $\varphi(t)$ by $\varphi(t)e^{-t/2}$ in \eqref{startingPoint} we obtain the following more convenient version of the trace formula:
\begin{equation}\label{startingPoint_0}
\sum_{s\in \mathcal{R}(Y)} \wh\varphi\left( i s \right) = -\frac{\vol_0(Y)}{4\pi} \int_{-\infty}^{\infty} \frac{\cosh(\frac{t}{2})}{\sinh(\frac{t}{2})^2}  e^{-\frac{t}{2}} \varphi(t)dt + \sum_{\ell\in  \mathcal{L}(Y)}\sum_{k=1}^\infty \frac{\ell \varphi(k\ell)}{1-e^{-k\ell}}.
\end{equation}
Resonances always come in conjugate pairs: if $s\in \mathcal{R}(Y)$, then also $\overline{s}\in \mathcal{R}(Y)$. This implies that whenever $\varphi$ is real-valued, then the sum over resonances on the left of \eqref{startingPoint_0} is real. If we assume further that $\varphi \geqslant 0$, then the integral on the right is positive, so we obtain the inequality
\begin{equation}\label{startingPoint_2}
\sum_{s\in \mathcal{R}(Y)} \wh\varphi\left( i s \right) \leqslant  \sum_{\ell\in  \mathcal{L}(Y)}\sum_{k=1}^\infty  \frac{\ell \varphi(k\ell)}{1-e^{-k\ell}}.
\end{equation}
Recall that Lemma \ref{lem:test_function} furnishes, for any $\beta > 1$, a non-negative function $\varphi = \varphi_\beta \in C_c^\infty(\mathbb{R})$ with $\supp(\varphi) = [-1,1]$ and
\begin{equation}\label{damping}
\vert \widehat{\varphi}(z) \vert \ll_\beta e^{\vert \mathrm{Im}(z)\vert} \exp\left( -c\frac{\vert \mathrm{Re}(z)\vert }{(\log \vert \mathrm{Re}(z)\vert)^{\beta}} \right), \quad \vert \mathrm{Re}(z)\vert \geqslant 2.
\end{equation}
For the rest of the proof we fix $\alpha > \beta > 1$ and we let $\varphi = \varphi_\beta$ be as above. Now we apply \eqref{startingPoint_2} to $Y = X_n$ and to the family of test functions
\begin{equation}\label{defiTestFunctions}
\varphi_{L,\eta}(x) \defeq \varphi\left( \frac{x-L}{\eta L} \right),
\end{equation} 
where $L$ and $\eta$ are positive parameters that we will choose further below, depending on $n$. By construction, $\varphi_{L,\eta}$ is a positive smooth function with support equal to
$$
\supp(\varphi_{L,\eta})  = [(1-\eta)L, (1+\eta)L].
$$
For the rest of this proof we assume that $0<\eta < 1 $ and that
$$
(1+\eta)L < \ell_0(X_n).
$$
This condition guarantees that the right of \eqref{startingPoint_2} equals zero, so
\begin{equation}\label{ineq0tr}
\sum_{s\in \mathcal{R}(X_n)} \wh\varphi_{L,\eta}\left( i s \right) \leqslant 0.
\end{equation}
Now we focus on the sum on the left. Given parameters $\sigma\in \mathbb{R}$ and $K> 0$ (to be chosen below) we partition the complex plane into the following four mutually disjoint regions:
\begin{itemize}
\item $R_1 = \{ \mathrm{Re}(s)>\frac{1}{2} \},$
\item $R_2 = \{ \text{$\sigma\leqslant \mathrm{Re}(s)\leqslant \frac{1}{2}$ and $\vert \mathrm{Im}(s)\vert\leqslant K$} \},$
\item $R_3 = \{ \text{$\sigma\leqslant \mathrm{Re}(s)\leqslant \frac{1}{2}$ and $\vert \mathrm{Im}(s)\vert > K$} \},$ and
\item $R_4 = \{ \mathrm{Re}(s)<\sigma\}.$
\end{itemize}
We divide the sum on the left of \eqref{ineq0tr} accordingly, writing
$$
\sum_{s\in \mathcal{R}(X_n)} \wh\varphi_{L,\eta}\left( i s \right) = S_1 + S_2 + S_3 + S_4,
$$
where 
$$
S_j = \sum_{s\in \mathcal{R}(X_n)\cap R_j} \wh\varphi_{L,\eta}\left( i s \right), \quad j\in \{ 1,2,3,4\}.
$$
Note that we have suppressed the dependence of the sums $S_j$ on $n,\eta,L,$ and $K$ since we assume these parameters to be fixed until the end of the proof. From \eqref{ineq0tr} we have
$$
S_1 + S_2 + S_3 + S_4 \leqslant 0.
$$
As we will see below, $S_1 >0$, so applying the triangle inequality gives
$$
S_1 \leqslant \vert S_2\vert + \vert S_3\vert + \vert S_4\vert. 
$$
Our goal now is to give a lower bound for $S_1$ and upper bounds for $S_2,S_3,$ and $S_4$. 

Note that by basic properties of the Fourier transform we have
$$
\wh\varphi_{L,\eta}\left( i s \right) = \eta L e^{sL} \wh\varphi( i \eta L s),
$$
which we will use repeatedly in the sequel.\\

\paragraph{\textit{Lower bound for $S_1$}}
Recall from the discussion in §\ref{sec:specth} that resonances in the half-plane $R_1 = \{ \mathrm{Re}(s)>\frac{1}{2} \}$ are related to eigenvalues $\lambda$ of the Laplacian via $\lambda = s(1-s)$:
$$
\mathcal{R}(X_n)\cap R_1 = \Omega(X_n) \subset [\frac{1}{2},\delta].
$$
It is easy to check that for all real $\xi$ we have $\wh\varphi\left( i \xi \right) \gg_\beta e^{-\xi}$ for some implied constant depending only on $\beta$. Hence, we have for all $s\in \Omega(X_n)$
$$
\wh\varphi_{L,\eta}\left( i s \right) = \eta L e^{sL} \wh\varphi( i \eta L s) \gg_\beta \eta L e^{sL(1-\eta)},
$$
and therefore,
\begin{equation}
S_1 = \sum_{s\in \Omega(X_n)} \wh\varphi_{L,\eta}\left( i s \right) \gg_\beta \eta L \sum_{s\in \Omega(X_n)} e^{sL(1-\eta)}.
\end{equation}
Note that $\lambda_0 = \delta(1-\delta)$ is the common base eigenvalue for all the surfaces $X_n$ (which by assumption are finite degree coverings of $X$). Thus, $\delta\in \Omega(X_n)$ for all $n\in \mathbb{N}$, whence
$$
S_1 \gg_\beta \eta L e^{\delta L(1-\eta)} .
$$

\paragraph{\textit{Upper bound for $S_2$}}
It is easy to verify that for all $s\in \CC$
\begin{equation}\label{triangleFourier}
\vert \wh\varphi_{L,\eta}\left( i s \right) \vert = \vert \int_{L(1-\eta)}^{L(1+\eta)}  \varphi_{L,\eta}(x) e^{sx} dx \vert \ll_{\beta} \eta L e^{ L(1 + \sign(\mathrm{Re}(s)) \eta ) \mathrm{Re}(s) }.
\end{equation}
Since all $s\in R_2$ have $\mathrm{Re}(s) \leqslant \frac{1}{2}$ we immediately obtain
$$
\vert S_2 \vert \ll_\beta \eta L e^{\frac{1}{2} L(1+\eta)}  \mathsf{N}(\sigma, K; X_n) .
$$

\paragraph{\textit{Upper bound for $S_3$}}
Here we use the decaying properties of the Fourier transform of $\varphi$. Assume furthermore that the parameters $\eta, L, K$ are chosen so that 
\begin{equation}\label{next_condotion}
\eta L K \geqslant 2.
\end{equation}
Then for all $s\in R_3$ we have $\eta L \vert \mathrm{Im}(s)\vert \geqslant 2$, so we may use the bound \eqref{damping} to obtain 
$$
\vert \wh\varphi_{L,\eta}\left( i s \right) \vert = \eta L e^{\mathrm{Re}(s) L} \vert \wh\varphi( i \eta L s) \vert \ll_\beta \eta L e^{\mathrm{Re}(s) L + \eta L \vert \mathrm{Re}(s)\vert }  \exp\left( - c \frac{\eta L \vert \mathrm{Im}(s)\vert}{( \log (\eta L \vert \mathrm{Im}(s)\vert) )^\beta} \right).
$$
Since all $s\in R_3$ have $\sigma \leqslant \mathrm{Re}(s) \leqslant \frac{1}{2}$ we get
\begin{equation}
\vert \wh\varphi_{L,\eta}\left( i s \right) \vert \ll_\beta \eta L e^{\frac{1}{2} L (1+\eta \max\{ 2\vert \sigma\vert , 1 \})}  f_\beta( \eta L \vert \mathrm{Im}(s)\vert),
\end{equation}
where $f_\beta$ is the function
\begin{equation}
f_\beta\colon [2, \infty) \to \mathbb{R}, \quad  f_\beta(t) = \exp\left( - c \frac{x}{( \log x )^\beta} \right).
\end{equation}
Therefore,
\begin{equation}\label{boundbeforestieltjes}
\vert S_3\vert \ll_\beta \eta L e^{\frac{1}{2} L (1+\eta \max\{ 2\vert \sigma\vert , 1 \})} \sum_{ s \in R_3 } f_\beta( \eta L \vert \mathrm{Im}(s)\vert).
\end{equation}
To estimate the sum on the left, recall from \eqref{pohl_soares} that we have the following bound for all $r\geqslant r_0(X)$:
$$
\mathsf{N}(r;X_n) \defeq \#\{ s\in\mathcal{R}(X_n) : \vert s\vert < r \} \ll \vol_0(X_n) r^2.
$$
Observe that $ \sigma \leqslant \mathrm{Re}(s)\leqslant \frac{1}{2}$ and $\vert \mathrm{Im}(s)\vert < t$ together imply $\vert s\vert \leqslant \sqrt{t^2 + \max\{ \vert \sigma\vert,\frac{1}{2} \}^2}$, so we have the rather crude bound
\begin{equation}\label{crude}
\mathsf{N}(\sigma, t; X_n) \leqslant \mathsf{N}\left( \sqrt{t^2 + \max\{ \vert \sigma\vert,\frac{1}{2} \}^2};X_n \right) \ll_\sigma \vol_0(X_n) t^2. 
\end{equation}
The remaining sum on the right of \eqref{boundbeforestieltjes} can be estimated with a Stieljes integral over the resonance counting function $\mathsf{N}(\sigma, t; X)$. Using integration by parts and the bound \eqref{crude} yields
\begin{align*}
\sum_{ s \in R_3 } f_\beta( \eta L \vert \mathrm{Im}(s)\vert) &\ll \int_{K}^{\infty} f_\beta (\eta L t ) \, d\mathsf{N}(\sigma, t; X_n)\\
&= \mathsf{N}(\sigma, K; X_n) f_\beta (\eta L K )  - \eta T \int_{K}^{\infty} \mathsf{N}(\sigma, t; X_n) f_\beta'( \eta L t ) \mathrm{d}t\\
&= \mathsf{N}(\sigma, K; X_n) f_\beta(\eta L K ) - \int_{K \eta T}^{\infty} \mathsf{N}(\sigma, \frac{u}{\eta L}; X_n) f_\beta' ( u ) \mathrm{d}u\\
&\ll_{\sigma} \vol_0(X_n) \left( K^{2} f_\beta (\eta L K )  -  \frac{1}{\eta^2 L^2}  \int_{K\eta L}^{\infty} u^{2} f_\beta' ( u ) \mathrm{d}u \right)
\end{align*}
It is an exercise to verify that for all $x\geqslant 2$
\begin{equation}\label{exercise}
-\int_{x}^{\infty} u^{2} f_\beta' ( u )  du \ll_\beta x^{3} f_\beta (x).
\end{equation}
Thus, returning to \eqref{boundbeforestieltjes}, we obtain
\begin{align*}
\vert S_3\vert &\ll_{\beta,\sigma} \eta L e^{\frac{1}{2} L (1+\eta \max\{ 2\vert \sigma\vert , 1 \})}   \sum_{ s \in R_3 } f_\beta ( \eta L \, \mathrm{Im}(s) )\\
&\ll_{\beta,\sigma}  \vol_0(X_n) \eta L e^{\frac{1}{2} L (1+\eta \max\{ 2\vert \sigma\vert , 1 \})} \left( K^{2}   + \eta L K^3  \right) f_\beta (\eta L K ) .
\end{align*}

\paragraph{\textit{Upper bound for $S_4$}}
Using the definition of the Fourier transform and repeated integration by parts, we obtain for all $m\in \mathbb{N}$
$$
\vert \wh\varphi_{L,\eta}\left( i s \right)\vert \ll_m \eta L \frac{e^{\mathrm{Re}(s) L(1+ \sign(\mathrm{Re}(s)) \eta) }}{\langle s\rangle^m}, \quad  \langle s \rangle \defeq \sqrt{1+\vert s\vert^2}.
$$
Thus, since all $s\in R_4$ have $\mathrm{Re}(s)\leqslant \sigma$ we get
$$
\vert S_4\vert \leqslant \sum_{s\in \mathcal{R}(X_n)\cap R_4} \vert \wh\varphi_{L,\eta}\left( i s \right) \vert \ll \eta L e^{\sigma L(1+  \sign(\sigma)  \eta)} \sum_{s\in \mathcal{R}(X_n) } \langle s\rangle^{-m}.
$$
We now take $m=3$ and we estimate the sum on the right with a Stieltjes integral
\begin{equation}
\sum_{s\in \mathcal{R}(X_n)} \langle s\rangle^{-3} = \int_0^{\infty}  \langle s\rangle^{-3}   d\mathsf{N}(r;X_n)
\end{equation}
over the global resonance counting function $\mathsf{N}(r;X_n)$. Using the bound \eqref{pohl_soares} yields
\begin{equation}
\sum_{s\in \mathcal{R}(X_n)} \langle s\rangle^{-3} \ll \vol_0(X_n)
\end{equation}
and therefore,
\begin{equation}
\vert S_4\vert \ll \vol_0(X_n) \eta L e^{\sigma L \left(1+  \sign(\sigma)  \eta \right)}.
\end{equation}
\\

\paragraph{\textit{Putting everything together}} 
Recall that 
$$
S_1 \leqslant \vert S_2\vert + \vert S_3\vert + \vert S_4\vert. 
$$
Combining the above estimates and writing $v_n = \vol_0(X_n)$ for notational convenience gives
\begin{align}
\eta L e^{\delta L(1-\eta)}  &\ll_{\beta,\sigma} \eta L e^{\frac{1}{2} L(1+\eta)}  \mathsf{N}(\sigma, K; X_n) + \label{final_bound_0}  \\
&\qquad + v_n \eta L e^{\frac{1}{2} L (1+\eta \max\{ 2\vert \sigma\vert , 1 \})} \left( K^{2}  + \eta L K^{3}  \right) f_\beta (\eta L K  )+ v_n \eta L e^{\sigma L(1+  \sign(\sigma)  \eta)} \nonumber,
\end{align}
provided the parameters $\eta, L,$ and $K$ satisfy the conditions
$$
0<\eta<1, \quad (1+\eta) L < \ell_0(X_n), \quad  \eta L K \geqslant 2.
$$
Dividing both sides of \eqref{final_bound_0} by $\eta L$ gives
\begin{align}
e^{\delta L(1-\eta)} &\ll_{\beta,\sigma} e^{\frac{1}{2}L(1+\eta)}  \mathsf{N}(\sigma, K; X_n) \label{final_bound} \\
&\qquad  + v_n e^{\frac{1}{2} L (1+\eta \max\{ 2\vert \sigma\vert , 1 \})} \left( K^{2}  +\eta L K^{3}  \right) f_\beta (\eta L K  ) + v_n e^{\sigma L(1+  \sign(\sigma)  \eta)} \nonumber.
\end{align}
Recall that by assumption we have as $n\to \infty$:
\begin{itemize}
\item $v_n \to \infty$ and
\item $\ell_0(X_n)\geqslant (A-o(1)) \log v_n$ for some constant $A>0$.
\end{itemize}
Now we fix some constant $\nu > 0$ satisfying $\alpha >\beta +\nu$ and we choose 
$$
\eta = \frac{1}{(\log \log v_n)^\nu}, \quad K = (\log \log v_n)^\alpha,
$$
and
$$
L = (A-w(n)) \log v_n
$$
for some suitable function $w(n)\to 0$ that ensures $(1+\eta) L < \ell_0(X_n)$. Clearly, we also have $0<\eta<1$ and $\eta L K \geqslant 2$ for all $n$ large enough. Recall that we have defined $f_\beta (x) = \exp\left( - c \frac{x}{( \log x )^\beta} \right) $, so one can check that for some constant $c' > 0$ we have
$$
f_\beta (\eta L K  ) \leqslant \exp\left( - c' (\log v_n) (\log \log v_n)^{\alpha - \beta - \nu}  \right).
$$
That is, our choices of $\nu, \eta, L,$ and $K$ imply that $f_\beta (\eta L K )$ decays faster as any polynomial in $v_n$. In particular, for any fixed $\sigma$, the second summand on the right of \eqref{final_bound} tends to zero as $n\to \infty$:
$$
v_n  e^{\frac{1}{2} L (1+\eta \max\{ 2\vert \sigma\vert , 1 \})} \left( K^{2} + \eta L K^{3}  \right) f_\beta (\eta L K  ) \to 0. 
$$
Therefore, for all $n$ large enough, the bound \eqref{final_bound} reduces to
\begin{equation}\label{toberearr}
e^{\delta L(1-\eta)} \ll_{\beta,\sigma} e^{\frac{1}{2}L(1+\eta)} \mathsf{N}(\sigma, K; X_n) + v_n e^{\sigma L(1+  \sign(\sigma)  \eta)}.
\end{equation}
Now we fix some $\epsilon'  >0$. By the above choices, we obtain that once
$$
\sigma \leqslant \delta -\frac{1}{A} - \epsilon',
$$
the term $v_n e^{\sigma L(1+  \sign(\sigma)  \eta)}$ becomes smaller than $e^{\delta L(1-\eta)}$ for all $n$ sufficiently large. Thus, upon rearranging \eqref{toberearr}, we obtain with $\sigma = \delta -\frac{1}{A} - \epsilon'$ the following lower bound:
$$
\mathsf{N}\left(\delta -\frac{1}{A}-\epsilon', K; X_n \right) \gg \frac{e^{\delta L(1-\eta)} - v_n e^{\sigma L(1+  \sign(\sigma)  \eta)} }{e^{\frac{1}{2}L(1+\eta)}} \gg e^{\delta L(1-\eta) - \frac{1}{2}L(1+\eta)}.
$$
It is now easy to check that for any fixed $\epsilon > 0$ and for all sufficiently large $n$ we have
$$
e^{\delta L(1-\eta) - \frac{1}{2}L(1+\eta)} = e^{ (\delta-\frac{1}{2})L - ( \delta+\frac{1}{2} )\eta L  } \gg v_n^{A ( \delta - \frac{1}{2} )-\epsilon},
$$
which completes the proof of Theorem \ref{thm:main_theorem}.

\section{Applying the main theorem to congruence covers}\label{sec:proof2}

We now give the proof of Theorem \ref{thm:congr}. For the rest of this section, we fix a non-elementary, convex cocompact subgroup $\Gamma\subset \mathrm{SL}_2(\ZZ)$. Note that the (principal) congruence subgroup $\Gamma(n)$ is nothing else but the kernel $\ker(\pi_n)$ of the reduction modulo $n$ map
$$
\pi_n \colon \Gamma \to \mathrm{SL}_2(\ZZ/n\ZZ),\quad \gamma\mapsto \gamma\mod n.
$$
By the result of Bourgain--Varj\'{u} \cite{Bourgain_Varju} there exists an integer $n_0$ such that for all $n$ with $\gcd(n,n_0)=1$ the map $\pi_n$ is surjective. In what follows, we assume that $\gcd(n,n_0)=1$, so $\Gamma(n)$ is a finite-index subgroup of $\Gamma$ of index
$$
[\Gamma:\Gamma(n)] = \# \mathrm{SL}_2(\ZZ/n\ZZ) \asymp n^3,
$$
and therefore, the Euler-characteristic of $X(n)$ is
$$
\vert \chi(X(n))\vert = [\Gamma:\Gamma(n)] \vert \chi(X)\vert \asymp n^3 \vert \chi(X)\vert.
$$ 
Using the Gauss-Bonnet relation $\vol_0(X) = -2\pi \chi(X)$ leads to
\begin{equation}\label{voln3}
\vol_0(X(n)) \asymp n^3,
\end{equation}
with implied constant depending only on the base surface $X.$ 

Now let us give a lower bound for the length of the shortest geodesic on $X(n)$. It is well-known that the set of closed geodesics on $ Y \cong \Pi\backslash \mathbb{H}^2 $ is bijective to the set $[\Pi]$ of conjugacy classes $[\gamma]$ of the primitive hyperbolic elements $ \gamma\in \Pi $, and the length of the geodesic corresponding to the conjugacy class of $\gamma \in \Pi$ equals the displacement length $\ell(\gamma)$, see for instance \cite{Borthwick_book}. Thus we need to give a lower bound for $\ell(\gamma)$ for all hyperbolic $\gamma\in \Gamma(n)$. To that effect, recall that the displacement length is given by the relation
\begin{equation}\label{eq_displacement}
\vert \mathrm{tr}(\gamma) \vert = 2 \cosh\left( \frac{\ell(\gamma)}{2} \right) < 2 e^{\frac{\ell(\gamma)}{2}},
\end{equation}
so it suffices to give a lower bound for the trace. For the rest of this section we write
$$
\gamma = \pmat{a}{b}{c}{d}, \quad a,b,c,d\in \ZZ, \quad ad-bc = 1.
$$
We write $I=\pmat{1}{0}{0}{1}$ for the identity, and we put $\Gamma^\ast = \Gamma\smallsetminus \{ I\}$. Suppose $\gamma\in \Gamma(n)^\ast$ for some $n\geqslant 2$, or equivalently, suppose that $\gamma$ satisfies the congruence
\begin{equation}\label{congruence_pm}
\gamma \equiv I \mod n.
\end{equation}
Then, using an argument of Sarnak--Xue \cite{SarnakXue}, we can show that the trace of $\gamma$ satisfies
\begin{equation}\label{non_trivial_congreunce}
\tr(\gamma) = a+d \equiv 2 \mod n^{2}.
\end{equation}
To see this, note that the congruence \eqref{congruence_pm} implies that there are integers $ a', b', c', d'\in \ZZ $ such that
$$
\text{$ a = a'n + 1 $, $ b = b'n $, $ c = c'n $, and $ d = d'n + 1 $.}
$$
Furthermore, the condition $ad-bc = 1$ gives
\begin{equation}
1 = (a'd' -b'c') n^{2} + (a'+d')n + 1,
\end{equation}
which forces
$$
a'+d' = 0 \mod n.
$$
But this implies that
$$
\mathrm{tr}(\gamma) = (a'+d' )n + 2 = 2 \mod n^2,
$$ 
as claimed. Now since $\gamma$ is a hyperbolic element we have $\vert \mathrm{tr}(\gamma)\vert  > 2$, which when combined with \eqref{non_trivial_congreunce} implies that for all $n\geqslant 2$ 
\begin{equation}\label{lowerBoundForTrace}
\vert \tr(\gamma)\vert \geqslant   n^{2}-2 \gg   n^2.
\end{equation}
Combining this with \eqref{eq_displacement} we obtain that
$$
\ell_0(X(n)) = \min_{\gamma\in \Gamma(n)^\ast} \ell(\gamma) \geqslant 2 \min_{\gamma\in \Gamma(n)^\ast} \log \left( \frac{\vert \mathrm{tr}(\gamma)\vert}{2} \right) \geqslant (4-o(1))\log(n).
$$
Invoking the bound \eqref{voln3}, this gives
$$
\ell_0(X(n)) \geqslant \left(\frac{4}{3}-o(1)\right) \log (\vol_0(X(n))).
$$
Thus, we can apply our main Theorem \ref{thm:main_theorem} with $A=\frac{4}{3}$: by increasing $n_0$ if necessary, we obtain that for all $n$ coprime to $n_0$ we have
$$
\mathsf{N}\left( \delta - \frac{3}{4} - \epsilon', (\log \log n)^\alpha; X(n) \right) \gg_{\alpha,\epsilon,\epsilon'}  \vol_0(X(n))^{\frac{4}{3} ( \delta-\frac{1}{2} )-\epsilon} \gg_{\alpha,\epsilon,\epsilon'} n^{4 ( \delta-\frac{1}{2} )-\epsilon}.
$$
This completes the proof of Theorem \ref{thm:congr}.

\normalem
\bibliography{low_lying_resonances} 
\bibliographystyle{amsplain}

\end{document}